\documentclass[amssymb, 11pt]{amsart}
\usepackage{latexsym}
\usepackage{young}

\newcommand{\qb}[2]{{\left [{#1 \atop #2} \right]_q}}
\newlength{\standardunitlength}
\newcommand{\bea}{\begin{eqnarray}}
\newcommand{\ena}{\end{eqnarray}}
\newcommand{\beas}{\begin{eqnarray*}}
\newcommand{\enas}{\end{eqnarray*}}

\newcommand{\ignore}[1]{}

\newtheorem{prop}{Proposition}[section]

\newtheorem{theorem}[prop]{Theorem}

\begin{document}

\title [Hall-Littlewood polynomials and Cohen-Lenstra heuristics] {Hall-Littlewood polynomials and Cohen-Lenstra heuristics
for Jacobians of random graphs}

\author{Jason Fulman}
\address{Department of Mathematics\\
        University of Southern California\\
        Los Angeles, CA, 90089, USA}
\email{fulman@usc.edu}

\keywords{Hall-Littlewood polynomial, Cohen-Lenstra heuristic, random graph}

\date{March 2, 2014}

\thanks{This research was supported by NSA grant H98230-13-1-0219. We also thank
Melanie Matchett Wood for helpful correspondence.}

\begin{abstract} Cohen-Lenstra heuristics for Jacobians of random graphs give rise to random
partitions. We connect these random partitions to the Hall-Littlewood polynomials of symmetric
function theory, and use this connection to give combinatorial proofs of properties of these
random partitions. In addition, we use Markov chains to give an algorithm for generating these
partitions.
\end{abstract}

\maketitle

\section{Introduction} \label{intro}

To a graph $\Gamma$, one can associate an abelian group $S_{\Gamma}$, which has various names in the literature: the sandpile
group, the Jacobian, the critical group, the Picard group. The size of $S_{\Gamma}$ is the number of spanning trees of $\Gamma$.
Lorenzini \cite{L} asked about the distribution of group structures
of $S_{\Gamma}$, for random graphs $\Gamma$. The papers \cite{CLP} and \cite{CKLPW}, proposed Cohen-Lenstra heuristics for
the distribution of $S_{\Gamma}$.

Matchett Wood \cite{MW} gives very concrete results about the $p$-Sylow subgroup $S_{\Gamma,p}$ of $S_{\Gamma}$. This is a finite abelian $p$-group. To
describe her work, recall that finite abelian $p$-groups correspond to integer partitions $\lambda$;
the group $\oplus \mathbb{Z}/p^{\lambda_i} \mathbb{Z}$ corresponds to the partition with parts $\lambda_1 \geq \lambda_2 \geq \cdots$.
Let $\Gamma$ be an Erdos-Renyi random graph on $n$ vertices with independent edge probabilities $0 < q <1$. Theorem 1.1 of \cite{MW} shows that
as $n \rightarrow \infty$, the chance that $S_{\Gamma,p}$ is equal to an abelian $p$-group of type $\lambda$ converges to
\begin{equation} \label{woodform}  \frac{\prod_{i \ odd} (1-1/p^i)}{p^{\sum_i \mu_i(\mu_i+1)/2} \prod_{i=1}^{\lambda_1} \prod_{j=1}^{\lfloor \frac{\mu_i - \mu_{i+1}}{2} \rfloor}
(1-1/p^{2j})}. \end{equation} Here $\mu$ is the partition obtained by transposing the diagram of the partition $\lambda$. Note that equation \eqref{woodform} does not depend on the edge probability $q$ of the graph.

Although Matchett Wood's result implies that formula \eqref{woodform} must define a probability distribution on the set of all partitions of
all natural numbers, this is not at all obvious from looking at formula \eqref{woodform}. One of our contributions is to relate
formula \eqref{woodform} to Hall-Littlewood polynomials, thereby implying that \eqref{woodform} defines a probability distribution.
As a first step, note that \eqref{woodform} can be rewritten as:
\begin{equation} \label{wood2} \frac{ \prod_{i \ odd} (1-1/p^i) } {p^{n(\lambda)+|\lambda|} \prod_{i \geq 1} \prod_{j=1}^{\lfloor m_i(\lambda)/2 \rfloor}
(1-1/p^{2j})}. \end{equation} Here $|\lambda|$ denotes the size of $\lambda$, $n(\lambda) = \sum_{i \geq 1} (i-1) \lambda_i$, and $m_i(\lambda)$
is the number of parts of $\lambda$ of size $i$. Next, consider the following identity on page 232 of \cite{M}:
\begin{equation} \label{macid} \sum_{\lambda} d_{\lambda}(t)^{-1} Q_{\lambda}(x;t) = \prod_i \frac{1-tx_i}{1-x_i} \prod_{i<j} \frac{1-tx_ix_j}{1-x_ix_j}, \end{equation}
where $Q_{\lambda}(x;t)$ is a family of symmetric functions which are scalar multiples of Hall-Littlewood polynomials, and
\[ d_{\lambda}(t) = \prod_{i \geq 1} \prod_{j=1} ^{\lfloor m_i(\lambda)/2 \rfloor} (1 - t^{2j}).\] Setting $x_i=1/p^i$ for all $i$, and $t=1/p$, it follows from page 213 of \cite{M} that $Q_{\lambda}(x;t)$ becomes $\frac{1}{p^{n(\lambda)+|\lambda|}}$. Hence with these values of the $x_i$'s and $t$, the left hand side of \eqref{macid} becomes \[
\sum_{\lambda} \frac{1}{p^{n(\lambda)+|\lambda|} \prod_{i \geq 1} \prod_{j=1}^{\lfloor m_i(\lambda)/2 \rfloor} (1-1/p^{2j})}.\] The right hand side of \eqref{macid} becomes
\begin{eqnarray*}
 \frac{1}{(1-1/p)} \prod_{i<j} \frac{(1-1/p^{i+j+1})}{(1-1/p^{i+j})} & = & \frac{1}{1-1/p} \frac{\prod_{s \geq 4} (1-1/p^s)^{\lfloor (s-2)/2 \rfloor}}{\prod_{s \geq 3} (1-1/p^s)^{\lfloor (s-1)/2 \rfloor}} \\
 & = & \prod_{i \ odd} (1-1/p^i)^{-1},
 \end{eqnarray*} giving a combinatorial proof that formula \eqref{woodform} defines a probability measure.

One advantage to the formulation in terms of Hall-Littlewood polynomials is that it suggests a two parameter family of probability measures
generalizing \eqref{woodform}. Here we have in mind a parameter $u$ (satisfying $0<u<p$), and a natural number $r$. Then one can specialize
the identity \eqref{macid} by setting $x_1=u/p, x_2 =u/p^2, \cdots, x_r = u/p^r$, all other $x_i=0$, and $t=1/p$.
Section \ref{Hall} of this paper will study two further specializations (the first being $u$ arbitrary and $r \rightarrow \infty$, and the second
being $u=1$ and $r$ arbitrary). Aside from their intrinsic interest, these two specializations will be useful. The first specialization will be used
to find a formula (which we believe to be new) for the chance that a partition $\lambda$ chosen from the probability measure \eqref{woodform} has
a given size. The second specialization will be used to give a combinatorial proof of the fact (proved algebraically by Matchett Wood \cite{MW}), that the chance
that $\lambda$ chosen from \eqref{woodform} has exactly $a$ parts is equal to \begin{equation} \label{parts} \frac{\prod_{i=a+1}^{\infty} (1-1/p^i)}
{p^{a(a+1)/2} \prod_{i=1}^{\infty} (1-1/p^{2i})}. \end{equation}

In retrospect, it is perhaps not so surprising that Hall-Littlewood polynomials should be usefully connected to Cohen-Lenstra heuristics for Jacobians of random graphs.
Indeed, in the case of number fields, the random partitions arising from Cohen-Lenstra heuristics for class groups were studied from the viewpoint of Hall-Littlewood polynomials in \cite{F2}, and this was our main motivation for the current paper. Moreover in the case of elliptic curves, one can easily deduce from formulas of Delaunay \cite{D} that the random partitions arising from Cohen-Lenstra heuristics for Tate-Shafarevich groups are captured by the random partitions in \cite{F2}. We also mention a nice recent paper of Delaunay and Jouhet \cite{DJ} which uses combinatorics of Hall-Littlewood polynomials to relate a model of Poonen and Rains to Cohen-Lenstra heuristics for Tate-Shafarevich groups.

Section \ref{Markov} of this paper uses Markov chains to study the probability measure \eqref{woodform}. More precisely, we give an algorithm for generating
such random partitions, analogous to one in \cite{F1} for random partitions arising from Cohen-Lenstra number field heuristics. This algorithm should be useful for
studying properties of the random partitions of \eqref{woodform} by computer simulations. Moreover, as a consequence of our methods, we get completely elementary proofs (using nothing about Hall-Littlewood polynomials) both for the fact that \eqref{woodform} defines a probability distribution, and for the fact that \eqref{parts} gives the chance
that $\lambda$ chosen from \eqref{woodform} has exactly $a$ parts.

\section{Hall-Littlewood polynomials} \label{Hall}

Hall-Littlewood polynomials appear in many places in algebraic combinatorics, for example in the representation theory of $GL(n,q)$ and as spherical functions
for $GL(n)$ over the $p$-adic numbers. All needed background on Hall-Littlewood polynomials can be found in Chapter 3 of the text \cite{M}. We also highly
recommend the survey paper \cite{NR}. For connections of Hall-Littlewood polynomials to random partitions (different than the random partitions studied here),
see \cite{F2}.

\subsection{First specialization} \label{First}

This subsection considers the specialization of identity \eqref{macid} with $x_i=u/p^i$ for all $i$, and $t=1/p$. The left hand side of
\eqref{macid} becomes
\[ \sum_{\lambda} \frac{u^{|\lambda|}} {p^{n(\lambda)+|\lambda|} \prod_{i \geq 1} \prod_{j=1}^{\lfloor m_i(\lambda)/2 \rfloor}
(1-1/p^{2j})}, \] and the right hand side becomes \begin{eqnarray*}
 \frac{1}{(1-u/p)} \prod_{i<j} \frac{(1-u^2/p^{i+j+1})}{(1-u^2/p^{i+j})} & = & \frac{1}{1-u/p} \frac{\prod_{s \geq 4} (1-u^2/p^s)^{\lfloor (s-2)/2 \rfloor}}{\prod_{s \geq 3} (1-u^2/p^s)^{\lfloor (s-1)/2 \rfloor}} \\
 & = & (1-u/p)^{-1} \prod_{i \geq 3 \atop i \ odd} (1-u^2/p^i)^{-1}.
 \end{eqnarray*} Thus we obtain the identity
\begin{equation} \label{size}
\sum_{\lambda} \frac{u^{|\lambda|}} {p^{n(\lambda)+|\lambda|} \prod_{i \geq 1} \prod_{j=1}^{\lfloor m_i(\lambda)/2 \rfloor}
(1-1/p^{2j})} = (1-u/p)^{-1} \prod_{i \geq 3 \atop i \ odd} (1-u^2/p^i)^{-1},
\end{equation} which implies that
\begin{equation} \label{sizemeas}
(1-u/p) \prod_{i \geq 3 \atop i \ odd} (1-u^2/p^i) \frac{u^{|\lambda|}} {p^{n(\lambda)+|\lambda|} \prod_{i \geq 1} \prod_{j=1}^{\lfloor m_i(\lambda)/2 \rfloor}
(1-1/p^{2j})} \end{equation}
defines a probability measure on the set of all partitions of all natural numbers.

As a corollary of this observation, we deduce the following result, which we believe to be new, and perhaps not so straightforward to
prove by other methods.

\begin{theorem} \label{sizen} If $\lambda$ is chosen from the probability measure \eqref{woodform}, then the chance that the size of $\lambda$ is equal
 to $n$ is \[ \frac{ \prod_{i \ odd} (1-1/p^i)}{p^n} \sum_{j=0 \atop j \ even}^n \frac{1}{p^{j/2} (1-1/p^2) (1-1/p^4) \cdots (1-1/p^j)} .\]
\end{theorem}

\begin{proof} The chance that the size of $\lambda$ is $n$ is equal to $\prod_{i \ odd} (1-1/p^i)$ multiplied by the coefficient of $u^n$ in
\[ \sum_{\lambda} \frac{u^{|\lambda|}} {p^{n(\lambda)+|\lambda|} \prod_{i \geq 1} \prod_{j=1}^{\lfloor m_i(\lambda)/2 \rfloor}
(1-1/p^{2j})}. \] By \eqref{size}, this is $\prod_{i \ odd} (1-1/p^i)$ multiplied by the coefficient of $u^n$ in \[ (1-u/p)^{-1} \prod_{i \geq 3 \atop i \ odd}
(1-u^2/p^i)^{-1} .\] Let $[u^s] f(u)$ denote the coefficient of $u^s$ in a power series $f(u)$. Then
\begin{eqnarray*}
& & [u^n] (1-u/p)^{-1} \prod_{i \geq 3 \atop i \ odd} (1-u^2/p^i)^{-1} \\
& = & \sum_{j=0}^n [u^{n-j}] (1-u/p)^{-1} \cdot [u^j] \prod_{i \geq 3 \atop i \ odd} (1-u^2/p^i)^{-1} \\
& = & \sum_{j=0}^n \frac{1}{p^{n-j}} [u^j] \prod_{i \geq 3 \atop i \ odd} (1-u^2/p^i)^{-1} \\
& = & \sum_{j=0 \atop j \ even}^n \frac{1}{p^{n-j}} [u^{j/2}] \prod_{i \geq 3 \atop i \ odd} (1-u/p^i)^{-1} \\
& = & \sum_{j=0 \atop j \ even}^n \frac{1}{p^{n-j}} \frac{1}{p^{3j/2} (1-1/p^2)(1-1/p^4) \cdots (1-1/p^j)} \\
& = &   \frac{1}{p^n} \sum_{j=0 \atop j \ even}^n \frac{1}{p^{j/2} (1-1/p^2) (1-1/p^4) \cdots (1-1/p^j)}.
\end{eqnarray*} Note that the fourth equality used the identity
\[  1 + \sum_{m \geq 1} \frac{s^m}{(1-q) \cdots (1-q^m)} = \prod_{m \geq 0} (1-s q^m)^{-1} \] (with $s=u/p^3$, $q=1/p^2$)
from page 19 of \cite{A}. This completes the proof. \end{proof}

\subsection{Second specialization} \label{Second}

Throughout this subsection we let $P(a)$ denote the probability that a partition chosen from \eqref{woodform} has exactly $a$ parts. As mentioned earlier,
an explicit formula for $P(a)$ is given by \eqref{parts}, and this subsection will give a combinatorial proof of this result. Matchett Wood's proof was algebraic, and
required a result of Macwilliams which enumerated the symmetric $n \times n$ matrices over a finite field according to their rank.

Consider the specialization of \eqref{macid} with $x_1=1/p, x_2=1/p^2, \cdots, x_r=1/p^r$, all other $x_i=0$, and $t=1/p$. From page 213 of \cite{M}, letting $l(\lambda)$ denote the number of parts of $\lambda$, it follows that the left hand side of \eqref{macid} becomes
\begin{equation} \label{left} \sum_{\lambda \atop l(\lambda) \leq r} \frac{1}{p^{n(\lambda)+|\lambda|} \prod_{i \geq 1} \prod_{j=1}^{\lfloor m_i(\lambda)/2 \rfloor} (1-1/p^{2j})}
    \cdot \frac{(1-1/p) \cdots (1-1/p^r)}{(1-1/p) \cdots (1-1/p^{r-l(\lambda)})} .\end{equation} One proves by induction on $r$ that the right hand side of \eqref{macid} specializes to
\[ (1+1/p)(1+1/p^2) \cdots (1+1/p^r).\] Summarizing, it follows that
\begin{equation} \label{partsmeas} \frac{(1+1/p)^{-1} \cdots (1+1/p^r)^{-1}}{p^{n(\lambda)+|\lambda|} \prod_{i \geq 1} \prod_{j=1}^{\lfloor m_i(\lambda)/2 \rfloor} (1-1/p^{2j})} \cdot \frac{(1-1/p) \cdots (1-1/p^r)}{(1-1/p) \cdots (1-1/p^{r-l(\lambda)})} \end{equation} defines a probability measure on the set of all
   partitions of all natural numbers with at most $r$ parts.

Now we prove the following theorem of Matchett Wood \cite{MW}.

\begin{theorem}
\[ P(a) = \frac{\prod_{i=a+1}^{\infty} (1-1/p^i)} {p^{a(a+1)/2} \prod_{i=1}^{\infty} (1-1/p^{2i})} .\]
\end{theorem}

\begin{proof} Note that \eqref{left} can be rewritten as
\begin{eqnarray*}
&  & (1-1/p) \cdots (1-1/p^r) \sum_{s=0}^r \frac{1}{(1-1/p) \cdots (1-1/p^{r-s})} \\
& & \cdot \sum_{\lambda: l(\lambda)=s}
\frac{1}{p^{n(\lambda)+|\lambda|} \prod_{i \geq 1} \prod_{j=1}^{\lfloor m_i(\lambda)/2 \rfloor} (1-1/p^{2j})} \\
& = & (1-1/p) \cdots (1-1/p^r) \sum_{s=0}^r \frac{1}{(1-1/p) \cdots (1-1/p^{r-s})} \frac{P(s)}{\prod_{i \ odd} (1-1/p^i)}
\end{eqnarray*} This gives the equation
\begin{eqnarray*}
& & (1-1/p) \cdots (1-1/p^r) \sum_{s=0}^r \frac{1}{(1-1/p) \cdots (1-1/p^{r-s})} \frac{P(s)}{\prod_{i \ odd} (1-1/p^i)} \\
& = & (1+1/p)(1+1/p^2) \cdots (1+1/p^r), \end{eqnarray*} which starting with $P(0)= \prod_{i \ odd} (1-1/p^i)$, one can use to recursively solve
for $P(1)$, then $P(2)$, and so on.

Since this recursion has a unique solution, it is enough to show that \[ P(a) = \frac{\prod_{i=a+1}^{\infty} (1-1/p^i)} {p^{a(a+1)/2} \prod_{i=1}^{\infty} (1-1/p^{2i})} \]
does satisfy the recursion. And indeed,
\begin{eqnarray*}
& & (1-1/p) \cdots (1-1/p^r) \\
& & \cdot \sum_{s=0}^r \frac{1}{p^{{s+1 \choose 2}} (1-1/p) \cdots (1-1/p^s) (1-1/p) \cdots (1-1/p^{r-s})} \\
& = & (1+1/p)(1+1/p^2) \cdots (1+1/p^r). \end{eqnarray*} Here the final equality used the $q$-binomial formula (page 78 of \cite{B})
\[ \sum_{s=0}^r \qb{r}{s} q^{s(s+1)/2} x^s = (1+xq)(1+xq^2) \cdots (1+xq^r) \] with the substitutions $x=1, q=1/p$.
This completes the proof. \end{proof}

\section{Markov chains} \label{Markov}

This section describes the probability measure \eqref{woodform} in terms of Markov chains. Two proofs of our main result (Theorem \ref{main}) will be given. The first proof is given in the interest
of clarity, and assumes that \eqref{woodform} is a probability measure and that a formula for $P(a)$ is given by \eqref{parts}. The second proof is more elementary,
and as a byproduct yields both another proof that \eqref{woodform} is a probability measure and another proof of the formula \eqref{parts} for $P(a)$.

It is convenient to set $\lambda_0'$ (the height of an imaginary zeroth column) equal to $\infty$. We let $Prob(E)$ denote the probability of an event
$E$ under the measure \eqref{woodform}.

\begin{theorem} \label{main} Starting with $\lambda_0'= \infty$, define in succession $\lambda_1',\lambda_2',\cdots$ according to the rule that if
$\lambda_i'=a$, then $\lambda_{i+1}'=b$ with probability
\[ K(a,b) = \frac{\prod_{i=1}^a (1-1/p^i)}{p^{{b+1 \choose 2}} \prod_{i=1}^b (1-1/p^i) \prod_{j=1}^{\lfloor (a-b)/2 \rfloor} (1-1/p^{2j}) }. \] Then
the resulting partition is distributed according to \eqref{woodform}.
\end{theorem}

\begin{proof} (First proof) If $\lambda$ is chosen from \eqref{woodform}, the chance that $\lambda_i'=r_i$ for all $i$ is
\begin{eqnarray*}
& & Prob(\lambda_0'=\infty) \frac{Prob(\lambda_0'=\infty, \lambda_1'=r_1)}{Prob(\lambda_0'=\infty)} \\
 & & \cdot \prod_{i=1}^{\infty} \frac{Prob(\lambda_0'=\infty,\lambda_1'=r_1,\cdots,\lambda_{i+1}'=r_{i+1})}{Prob(\lambda_0'=\infty,\lambda_1'=r_1,\cdots,\lambda_{i}'=r_{i})}.
 \end{eqnarray*}
Thus it is enough to prove that \begin{eqnarray*}
 & & \frac{Prob(\lambda_0'=\infty,\lambda_1'=r_1,\cdots,\lambda_{i-1}'=r_{i-1},\lambda_i'=a,\lambda_{i+1}'=b)} {Prob(\lambda_0'=\infty,\lambda_1'=r_1,\cdots,\lambda_{i-1}'=r_{i-1},\lambda_i'=a)} \\
 & = & \frac{\prod_{i=1}^a (1-1/p^i)}{p^{{b+1 \choose 2}} \prod_{i=1}^b (1-1/p^i) \prod_{j=1}^{\lfloor (a-b)/2 \rfloor} (1-1/p^{2j}) },
 \end{eqnarray*} for all $i,a,b,r_1,\cdots,r_{i-1} \geq 0$.

For the case $i=0$, one needs to check that \[ P(b) = \frac{\prod_{i \ odd} (1-1/p^i)}{p^{{b+1 \choose 2}} \prod_{i=1}^b (1-1/p^i)},\] and this follows from
\eqref{parts}. For $i>0$, one calculates that
\begin{eqnarray*}
& & \sum_{\lambda:\lambda_1'=r_1, \cdots \lambda_{i-1}'=r_{i-1} \atop \lambda_i'=a} Prob(\lambda) \\
& = & \frac{P(a)}{p^{{r_1+1 \choose 2}+\cdots+{r_{i-1}+1 \choose 2}}} \\
& & \cdot \frac{1}{(1/p^2)_{\lfloor (r_1-r_2)/2 \rfloor} \cdots (1/p^2)_{\lfloor (r_{i-2}-r_{i-1})/2 \rfloor} (1/p^2)_{\lfloor (r_{i-1}-a)/2 \rfloor}}, \end{eqnarray*}
where $(1/p^2)_j$ denotes the product $(1-1/p^2)(1-1/p^4) \cdots (1-1/p^{2j})$. Similarly, observe that

\begin{eqnarray*} & &  \sum_{\lambda:\lambda_1'=r_1, \cdots \lambda_{i-1}'=r_{i-1} \atop \lambda_i'=a,\lambda_{i+1}'=b} Prob(\lambda) \\
& = & \frac{P(b)}{p^{{r_1+1 \choose 2}+\cdots+{r_{i-1}+1 \choose 2}+{a+1 \choose 2}}} \\
& & \cdot \frac{1}{(1/p^2)_{\lfloor (r_1-r_2)/2 \rfloor} \cdots (1/p^2)_{\lfloor (r_{i-1}-a)/2 \rfloor }(1/p^2)_{\lfloor (a-b)/2 \rfloor }}.
\end{eqnarray*}

Thus the ratio of these two expressions is \[ \frac{P(b)}{p^{{a+1 \choose 2}} P(a) (1/p^2)_{\lfloor (a-b)/2 \rfloor}} = \frac{\prod_{i=1}^a (1-1/p^i)}{p^{{b+1 \choose 2}} \prod_{i=1}^b (1-1/p^i) \prod_{j=1}^{\lfloor (a-b)/2 \rfloor} (1-1/p^{2j})}, \] as desired. Note that the transition probabilities automatically sum to 1 because
\[ \sum_{b \leq a} \frac{\sum_{\lambda:\lambda_1'=r_1, \cdots \lambda_{i-1}'=r_{i-1} \atop \lambda_i'=a,\lambda_{i+1}'=b} Prob(\lambda)}{\sum_{\lambda:\lambda_1'=r_1, \cdots \lambda_{i-1}'=r_{i-1} \atop \lambda_i'=a} Prob(\lambda)} = 1 \] for any measure on partitions.
\end{proof}

\begin{proof} (Second proof) For this proof one only needs that \eqref{woodform} is a measure; it will emerge that it is a probability measure. We let
 $P(a)$ denote the mass that the measure \eqref{woodform} assigns to the event that $\lambda$ has $a$ parts, and do not assume a formula for $P(a)$. We also let
 $Prob(\lambda)$ denote the mass that the measure \eqref{woodform} assigns to $\lambda$.

 As in the first proof, one calculates that
\begin{eqnarray*}
& & \sum_{\lambda:\lambda_1'=r_1, \cdots \lambda_{i-1}'=r_{i-1} \atop \lambda_i'=a} Prob(\lambda) \\
& = & \frac{P(a)}{p^{{r_1+1 \choose 2}+\cdots+{r_{i-1}+1 \choose 2}}} \\
& & \cdot \frac{1}{(1/p^2)_{\lfloor (r_1-r_2)/2 \rfloor} \cdots (1/p^2)_{\lfloor (r_{i-2}-r_{i-1})/2 \rfloor} (1/p^2)_{\lfloor (r_{i-1}-a)/2 \rfloor}},
\end{eqnarray*} where
$(1/p^2)_j$ denotes the product $(1-1/p^2)(1-1/p^4) \cdots (1-1/p^{2j})$, and that
\begin{eqnarray*}
& & \sum_{\lambda:\lambda_1'=r_1, \cdots \lambda_{i-1}'=r_{i-1} \atop \lambda_i'=a,\lambda_{i+1}'=b} Prob(\lambda) \\
& & = \frac{P(b)}{p^{{r_1+1 \choose 2}+\cdots+{r_{i-1}+1 \choose 2}+{a+1 \choose 2}}} \\
& & \cdot \frac{1}{(1/p^2)_{\lfloor (r_1-r_2)/2 \rfloor} \cdots (1/p^2)_{\lfloor (r_{i-1}-a)/2 \rfloor }(1/p^2)_{\lfloor (a-b)/2 \rfloor }}. \end{eqnarray*}

The ratio of these two expressions is \[ \frac{P(b)}{p^{{a+1 \choose 2}} P(a) (1/p^2)_{\lfloor (a-b)/2 \rfloor}} .\] Since \eqref{woodform} defines a measure, it follows that \[ \sum_{b \leq a} \frac{P(b)}{p^{{a+1 \choose 2}} P(a) (1/p^2)_{\lfloor (a-b)/2 \rfloor}} = 1.\] From this recursion and the fact that $P(0) = \prod_{i \ odd} (1-1/p^i)$, one solves for $P(a)$ inductively, finding that \[ P(a) = \frac{\prod_{i=a+1}^{\infty} (1-1/p^i)}
{p^{a(a+1)/2} \prod_{i=1}^{\infty} (1-1/p^{2i})}.\] Now note that
\begin{eqnarray*}
\sum_a P(a) & = & \frac{1}{\prod_{i \geq 1} (1-1/p^{2i})} \sum_{a \geq 0} \frac{\prod_{i \geq a+1} (1-1/p^i)}{p^{{a+1 \choose 2}}} \\
& = & \frac{\prod_i (1-1/p^i)}{\prod_i (1-1/p^{2i})} \sum_{a \geq 0} \frac{1}{p^{{a+1 \choose 2}} (1-1/p) \cdots (1-1/p^a)} \\
& = & \frac{\prod_i (1-1/p^i)}{\prod_i (1-1/p^{2i})} \prod_{i} (1+1/p^i) \\
& = & 1, \end{eqnarray*} where the third equality used Identity 2.2.6 on page 19 of \cite{A}. It follows that \eqref{woodform} defines a probability measure.
\end{proof}

\end{document}